\documentclass[11pt, reqno]{amsart}
\usepackage{amsmath, amssymb, amsfonts, amstext, verbatim, amsthm}
\usepackage[all]{xy}
\usepackage{color}
\usepackage[modulo]{lineno}

\theoremstyle{plain}
\newtheorem{thm}{Theorem}[section]
\newtheorem{lem}{Lemma}
\newtheorem{claim}[thm]{Claim}
\newtheorem{cor}[thm]{Corollary}

\newtheorem{prop}[thm]{Proposition}

\theoremstyle{definition}
\newtheorem{rem}[thm]{Remark}
\newtheorem{defn}[thm]{Definition}
\newtheorem{ex}[thm]{Example}
\newtheorem{fact}[thm]{Fact}

{\end{minipage} \end{center}}

\newcommand{\X}{\mathfrak{X}}
\newcommand{\Z}{\mathfrak{Z}}
\newcommand{\LL}{\mathfrak{L}}

\newcommand{\pr}{\operatorname{pr}}
\newcommand{\id}{\operatorname{id}}

\newcommand{\OO}{\mathcal{O}}

\newcommand{\h}{\mathfrak{h}}
\newcommand{\del}{\partial}
\newcommand{\delb}{\overline{\partial}}

\newcommand{\real}[1]{\mathbb{B} #1 }

\newcommand{\LHT}{LHT}
\newcommand{\DM}{DM}
\newcommand{\ibid}{[ibid.]}

\newcommand{\Aut}{\operatorname{Aut}}
\newcommand{\ind}{\operatorname{ind}}
\newcommand{\abs}[1]{\left| #1 \right|}
\newcommand{\stack}[1]{\mathfrak{#1}}

\newcommand{\RR}{\mathbb{R}}
\newcommand{\CC}{\mathbb{C}}
\newcommand{\ZZ}{\mathbb{Z}}

\newcommand{\CP}{\mathbb{CP}}
\newcommand{\RP}{\mathbb{RP}}

\newcommand{\op}[1]{\operatorname{#1}}

\begin{document}

\title{Lefschetz Hyperplane Theorem for Stacks}
\author{Daniel Halpern-Leistner}

\begin{abstract}
We use Morse theory to prove that the Lefschetz Hyperplane Theorem holds for compact smooth Deligne-Mumford stacks over the site of complex manifolds.  For $\Z \subset \X$ a hyperplane section, $\X$ can be obtained from $\Z$ by a sequence of deformation retracts and attachments of high-dimensional finite disc quotients.  We use this to derive more familiar statements about the relative homotopy, homology, and cohomology groups of the pair $(\X,\Z)$.  We also prove some preliminary results suggesting that the Lefschetz Hyperplane Theorem holds for Artin stacks as well.  One technical innovation is to reintroduce an inequality of {\L}ojasiewicz which allows us to prove the theorem without any genericity or nondegeneracy hypotheses on $\Z$.
\end{abstract}

\maketitle

\section{Introduction}


Consider an $n$-dimensional complex manifold $X$ with a positive line bundle $L$ and a global section $s$ with zero locus $Z \subset X$.  The Lefschetz hyperplane theorem (henceforth \LHT) states that up to homotopy equivalence $X$ is obtained from $Z$ by attaching cells of dimension $n$ or larger, or equivalently that $\pi_i(X,Z) = 0$ for $i<n$.  By the homological \LHT\ we refer to the weaker statement that $H_i(X,Z) = H^i(X,Z) = 0$ for $i<n$.

The theorem fails for analytic spaces or complex varieties which are not smooth, although it still holds for projective varieties as long as the compliment of the hyperplane section is a local complete intersection.  See the book \cite{GoreskyMacPherson88} for an exhaustive treatment of the classical \LHT\ and its generalizations.  In this note we go in a different direction and prove the theorem for smooth Deligne-Mumford stacks.

It is not surprising that with rational coefficients, the homological \LHT\ holds for a variety with orbifold singularities.  Poincar\'{e}-Lefschetz duality holds rationally for such varieties, so the fact that $X-Z$ is affine and thus $H_i(X-Z) = 0$ for $i>n$ implies the homological \LHT.  One of the main points of this note, then, is that the classical \LHT\ still holds with \emph{integral} coefficients (and in fact with homotopy groups) as long as one considers the homology of the stack instead of the underlying singular variety.

In Section \ref{LHT_DM_stacks} we show how the ideas of Bott's Morse theoretic proof of the \LHT\ \cite{Bott59} extend to the setting of \DM\ stacks.  We prove our main version of the \LHT, which states that a stack can be constructed from a hyperplane section by a series of deformation retracts and attachments of high-dimensional disc quotients.  Then we show how the usual statements about relative homotopy groups and homology groups follow.  The homotopy groups of a stack are defined to be the homotopy groups of its classifying space -- this is discussed in more detail below.

In Section \ref{morse_theory_DM_stacks} we go back and carefully develop Morse theory of \DM\ stacks.  The main lemmas of Morse theory have been proven for the \emph{underlying space} of a \DM\ stack in \cite{Hepworth09}, but we observe that the proofs work for the stack itself.  In this paper, a cell attachment refers to an honest 2-categorical colimit.  This is the strongest notion of gluing discussed in \cite{Noohi05} -- it doesn't exist unless one is gluing along embedded substacks, and even then some care is required.  Thus the main technical hurdle in Section \ref{morse_theory_DM_stacks} is showing that a Morse function on $\X$ really does give a cell decomposition in the strongest sense.

Finally, in Section \ref{homological_artin_LHT} we discuss the \LHT\ for smooth Artin stacks.  We prove a version of the homological \LHT\ for differentiable stacks admitting presentations in which both $X_1$ and $X_0$ are compact (real) manifolds.  We also present some calculations which suggest that the homological \LHT\ holds for a much larger class of Artin stacks.

Another important point in this note is a classical result on the stability of gradient descent flow.  In Bott's original paper, he assumes that the zero locus of $s$ is ``non-degenerate'' in order to apply Morse theory.  Here we prove the \LHT\ without any assumptions on the section.  The key observation is that the gradient flow of a function is stable under much weaker hypotheses than the function being non-degenerate -- as long as a function is real analytic (or close enough) its critical loci will have neighborhood deformation retracts.  These ideas were developed in the work of {\L}ojasiewicz\cite{Lojasiewicz84}, but it's a bit of real-analytic function theory that's been overlooked in the Morse theory literature.  We discuss some of {\L}ojasiewicz's results in Appendix \ref{app_lojasiewicz} and refer to these ideas often throughout the paper.

The technical foundation for Morse theory of DM stacks has been laid by Hepworth in \cite{Hepworth09}.  In that paper he develops Morse functions, Riemannian metrics, integration of vector fields, the strong topology on $C^\infty (\X)$, and the main theorem of Morse theory for the underlying space $\bar{\X}$.\footnote{Although he doesn't prove the main theorem on the level of stacks, which we do in Section \ref{morse_theory_DM_stacks}}  We will freely use these foundational constructions throughout this note without explicitly referencing them.

I'd like to thank my adviser, Constantin Teleman, for suggesting this project to me and for many useful conversations throughout.

\section{Hyperplane theorem for \DM\ stacks} \label{LHT_DM_stacks}

Let $\X$ be an $n$-dimensional Deligne-Mumford stack\footnote{Meaning $\X$ admits a surjective representable \'{e}tale map from a complex manifold $X_0 \to \X$, and $\Delta : \X \to \X \times \X$ is representable and proper. Equivalently $\X$ admits a presentation by a proper \'{e}tale Lie groupoid.} over the site of complex manifolds and $\LL$ a holomorphic line bundle on $\X$ with hermitian structure $\h$ of type $k$ (defined below).  $s \in \Gamma(\X,\LL)$ will be a section of $\LL$, and $\Z \subset \X$ will denote the zero locus of $s$.  We assume furthermore that $\bar{\X}$, the coarse moduli space of $\X$ treated as a topological stack, is compact.

Choose an \'{e}tale atlas $X_0 \to \X$ such that $\LL|_{X_0}$ admits a nowhere vanishing holomorphic section $\sigma$.  Then the $(1,1)$-form
$$\Theta_0 = \delb \del \log \h(\sigma, \sigma) \in \Omega^{1,1} (X_0)$$
descends to a hermitian global section of $\Omega^{1,1}_\X$, the curvature form $\Theta$.  In local analytic coordinates we can write $\Theta$ as $H_{\alpha \beta} dz^\alpha \wedge d\bar{z}^\beta$, and following \cite{Bott59} we call $(\LL,\h)$ \emph{type $k$} if the hermitian matrix $H_{\alpha \beta}$ is positive on a space of (complex) dimension $k$ at every point.

\begin{rem} \label{artin_type_k}
This notion, although not mainstream, is the right one from the perspective of descent.  $\Omega^{p,q}$ is a sheaf on the Artin site of complex manifolds, so the construction of $\Theta$ above works when $X_0 \to \X$ is a surjective submersion and not necessarily \'{e}tale.  In addition, the property of being positive definite on a $k$-dimensional subspace is Artin-local, i.e. if $\varphi : X \to Y$ is a surjective submersion, then a hermitian form on $Y$ is type $k$ iff its pullback to $X$ is type $k$.  Thus we can define a type $k$ holomorphic line bundle on an \emph{Artin} stack.
\end{rem}

Line bundles of type $k$ arise naturally from maps to $\CP^N$.  Let $\varphi : \X \to \CP^N$ be a holomorphic map and $a : X_0 \to \X$ an atlas, either \'{e}tale or submersive.  Define the minimal rank of $\phi$ to be the minimum over all $x \in X_0$ of the complex rank of $D_{x} (\phi \circ a) : T_x X_0 \to T_{\phi\circ a (x)} \CP^N$.  The minimal rank does not change if we pre-compose $a$ with another surjective submersion $X_0^\prime \to X_0$, and thus the minimal rank of $\varphi$ does not depend on the choice of atlas.

Choosing a hermitian form on $\CC^{n+1}$ induces a positive hermitian structure $\h$ on $\OO_{\CP^N}(1)$.  If $\phi : \X \to \CP^N$ has minimal rank $k$, then $\phi^\ast \OO(1)$ is type $k$ on $\X$, and a hyperplane $H \subset \CP^N$ gives a section of $\phi^\ast \OO(1)$.  In this case the \LHT\ says that the relative homotopy groups $\pi_i (\X, \varphi^{-1} (H))$, interpreted as the relative homotopy groups of the classifying spaces, vanish for $i<k$.  Thus one can think of the generalization of the \LHT\ to type $k$ line bundles as a ``large fiber'' generalization as discussed in \cite{GoreskyMacPherson88}.

\vskip 10 pt

Now we will prove the \LHT\ by analyzing the \lq{}\lq{}Morse\rq{}\rq{} function $f = \h(s,s) \in C^\infty_\RR (\X)$.  Define the closed topological substack $\X^{\epsilon]} = f^{-1} [0,\epsilon]$ and the open differentiable substack $\X^{ \epsilon)} = f^{-1} [0,\epsilon)$.  Note that $\X^{0]} = f^{-1} (\{0\}) = \Z$ is precisely the vanishing locus of $s$.

Fixing a Riemannian metric on $\X$, we first show that $f$ satisfies a {\L}ojasiewicz inequality near its global minimum.  This property of $f$ guarantees that the gradient descent flow of $f$ retracts $\X^{\epsilon]}$ onto $\Z$ for sufficiently small $\epsilon$, but we will not make this precise until Section \ref{morse_theory_DM_stacks}.
\begin{claim}
For a sufficiently small $\epsilon >0$, there are constants $C>0$ and $\rho \in (0,1)$ such that
\begin{equation*} \tag{L} \label{eqn_lojasiewicz}
|f|^\rho \leq C |\nabla f|
\end{equation*}
on the substack $\X^{\epsilon]}$.
\end{claim}

\begin{proof}
Note that it suffices to verify this inequality on $\bar{\X}^{\epsilon]}$.  By compactness of $\bar{\Z}$ it suffices to verify \eqref{eqn_lojasiewicz} around each $p \in \bar{\Z}$.  Choose an \'{e}tale analytic coordinate patch $U \to \X$ containing $p\in \bar{\Z}$ such that $\LL|_U$ is trivial.  Then $f_U = \h (s,s) = h \cdot \abs{s}^2$ is the product of a nonvanishing smooth function $h$ with the real analytic function $|s|^2$, and thus by Theorem \ref{app_stab_class} below the inequality \eqref{eqn_lojasiewicz} holds in a neighborhood of $p$.
\end{proof}

Next we will perturb $f$ to be Morse without affecting its values on $\X^{\epsilon/2]}$, so that our new function will continue to satisfy \eqref{eqn_lojasiewicz}.  Let $g : \X \to [0,1]$ be a smooth function such that $g |_{\X^{\epsilon/2]}} = 0$ and $g |_{\X - \X^{\epsilon)}} = 1$.

\begin{claim} \label{claim_perturbation}
There is a small open neighborhood $U \subset C^\infty (\X)$ containing $0$ such that for any $\eta \in U$, the new function $f^\prime = f + g \eta$ has the properties
\begin{itemize}
\item $f^\prime$ has no critical points on $\X^{\epsilon]} - \X^{\epsilon/2 )}$
\item $f^\prime$ is positive on $\X-\X^{\epsilon/2)}$
\item the hermitian form $\delb \del \log f^\prime$ is positive on a space of dimension $k$ at every point in $\X-\X^{\epsilon/2)}$.
\end{itemize}
\end{claim}

\begin{proof}
Multiplication by $g$ induces a continuous map $C^\infty \X \to C^\infty \X$, so it suffices to prove that the set of $\eta^\prime$ for which $f+\eta^\prime$ satisfies the above properties is open in $C^\infty \X$ and contains $0$.

Now choose a map $V \to \X$ and an open subset $X_0 \subset V$ such that $X_0 \to \X$ is surjective \'{e}tale and $\op{cl}(X_0) \to \X$ is proper.  Then $C^\infty(\X)$ is identified as a vector space with invariant functions in $C^\infty(X_0)$, and one can show that for such an atlas the strong topology on $C^\infty(\X)$ is the subspace topology in $C^\infty(X_0)$.  All three properties are open conditions on the derivatives of $(f+\eta)|_{X_0}$ up to order $2$, and they are satisfied by $f$, so the claim follows.

%
\end{proof}

Morse functions form a dense open subset of $C^\infty (\X)$, so there is some $\eta \in U$, where $U \subset C^\infty \X$ is defined in Claim \ref{claim_perturbation}, such that $f + \eta$ is Morse.  Consider the new function $f^\prime = f + g\eta$.  $f^\prime$ will satisfy all of the properties of Claim \ref{claim_perturbation}, and in addition it will be Morse on the substack $\X - \Z$, because $f^\prime$ contains no critical points in $\X^{\epsilon]}-\Z$ and $f^\prime$ agrees with $f+\eta$ on $\X-\X^{\epsilon)}$.  From this point on we replace our function $f$ with the new function $f^\prime$.

Next we consider the index of the critical points of $f$ away from $\Z$.  Recall that for a critical point $c$, the tangent space $T_c \X$ and the Hessian $H_c f$ are defined in the groupoid $U \times_{\X} U \rightrightarrows U$, where $U \to \X$ is an analytic coordinate patch containing $c$.  The tangent space splits (non-canonically) as a representation of $\Aut_c$, $T_c \X \simeq T_c \X_{-} \oplus T_c \X_{+}$ on which $H_c f$ is negative definite and positive definite respectively.  The index, $\ind_c$ is the isomorphism class of the representation $T_c \X_{-}$ of $\Aut_c$.

Of course in a local coordinate patch around $c$, the classical observation that the Hessian is negative definite on a space of dimension $k$ still applies.
\begin{claim} \cite{Bott59}
At every critical point $c \in \bar{X} - \bar{Z}$, the index representation has $\dim (\ind_c) \geq k$.
\end{claim}

Thus we have a function $f : \X \to \RR$ that is Morse on the open substack $\X - \Z$ and such that $\Z \subset \X$ is a \emph{stable} global minimum of $f$.  Now from the results of Morse theory in Section \ref{morse_theory_DM_stacks}, we have the main theorem:

\begin{thm} \label{thm_main_stacks}
As a topological stack, $\X$ can be obtained from $\Z$ by a finite sequence of deformation retracts and gluings of $[D^r / G]$ along $[S^{r-1}/G]$, where $r \geq k$.  Here we can assume for each cell attachment that $G = \Aut_c$ acting linearly on the unit disc in $T_c \X_{-}$ for some $c \in \bar{\X} - \bar{\Z}$.
\end{thm}

\begin{rem}
We only need $\LL$ to be type $k$ on the complement $\X -\Z$.  Geometrically, if $\LL$ arises from a map $\phi : \X \to \CP^N$, this means we only need $k$ to be the minimal rank of the restriction $\phi^{-1} (\CP^N - H) \to \CP^N - H$.
\end{rem}

We will define more carefully what a deformation retract and a cell attachment are below, but suffice it so say that many corollaries about the various topological invariants follow.

Singular homology and cohomology are defined from the singular chain bi-complex of a simplicial manifold presenting $\X$.  The reader is referred to \cite{Behrend04} and \cite{Behrend03} for an in-depth discussion and the equivalence between this and other notions of cohomology (i.e. de Rham or sheaf cohomology).

One thing that's not described in \cite{Behrend04} is that $H^i$ and $H_i$ are homotopy functors, but this follows from the usual argument, the key fact being the Poincar\'{e} lemma: by a local-to-global spectral sequence the projection $p : \X \times I \to \X$ induces an isomorphism in homology and cohomology.

The homotopy groups are a bit more subtle.  We define $\pi_i(\X) = \pi_i(\real{\X})$ for a classifying space $\real{\X}$.  One way to define the classifying space $\real{\X}$ is as the geometric realization of the simplicial manifold $X_p = X_0 \times_{\X} \cdots \times_{\X} X_0$ defined using an atlas $X_0 \to \X$ -- see \cite{Segal68} for a definition and properties of this construction.

A more recent discussion using the modern language of stacks appears in \cite{Noohi08}.  There the author uses the Haefliger-Milnor construction of a classifying space.  In this note all stacks will admit presentations in which $X_1$ and $X_0$ are both metrizable -- they are hoparacompact in the language of \cite{Noohi08} -- in which case the Segal construction will suffice.

In any event, we will only need a few properties of $\real{\X}$.  First there is an epimorphism $\varphi : \real{\X} \to \X$ which is a universal weak equivalence, meaning that the base change along any map $T \to \X$ is a weak equivalence.  $\varphi$ induces isomorphisms on homology and cohomology, in other words the homology of $\real{\X}$ agrees with the double complex homology defined above.\footnote{If $X_p$ is a simplicial presentation for $\X$ and $X^\prime_p$ the pulled-back presentation of $\real{\X}$, then $X^\prime_p \to X_p$ is a weak equivalence for all $p$, so the spectral sequence $E_1^{p,q} = H^q(X_p) \Rightarrow H^{p+q}(\X)$ implies that $H^\ast ( \X ) \to H^\ast (\real{\X})$ is an isomorphism.}

Also the construction $\real{\X}$ is a \emph{functor} from the category of (hoparacompact) stacks to the homotopy category of paracompact spaces.  Finally, for $\Z \subset \X$ an embedding, we can construct $\real{\Z}$ such that $\real{\Z} \subset \real{\X}$ is a subspace.  Thus we define the relative homotopy groups $\pi_i(\X,\Z) := \pi_i(\real{\X},\real{\Z})$.  We omit base points from our notation, but the following theorem applies for any base point in $\real{\Z}$.

\begin{cor} \label{cor_htpy_groups}
We have $\pi_i(\X,\Z) = 0$ for $i<k$.  Hence from the long exact sequence of a pair $\pi_i(\Z) \to \pi_i(\X)$ is an isomorphism for $i<k-1$ and surjective for $i=k-1$.
\end{cor}

\begin{proof}
For a triple of topological stacks $\Z \subset \X \subset \X^\prime$ we have a long exact sequence
$$\cdots \to \pi_i (\X,\Z) \to \pi_i(\X^\prime, \Z) \to \pi_i(\X^\prime, \X) \to \cdots$$
so if $\pi_i(\X^\prime,\X) = \pi_i(\X,\Z) = 0$, then $\pi_i(\X^\prime,\Z) = 0$ as well.  Thus the corollary follows by induction on the sequence of cell attachments and deformation retracts described in Theorem \ref{thm_main_stacks}, once we verify that $\pi_i(\X^\prime,\X) = 0$ for $i<k$ when $\X \subset \X^\prime$ is a disc attachment or deformation retract.


For a deformation retract, the inclusion $\real{\X} \to \real{\X^\prime}$ is a homotopy equivalence, so $\pi_i(\X^\prime, \X) = 0$ for all $i$.

Next let $\X^\prime = \X \cup_{[S/G]} [D/G]$ be a 2-categorical colimit, where $D$ is a disc of dimension $\geq k$, $S$ its boundary sphere, and $G$ a finite group acting linearly on the pair $(D,S)$.  Let $A \subset D$ be the complement of an open disc of half the radius of $D$, and consider the substack $\X \cup_{[S/G]} [A/G] \subset \X^\prime$.

First observe that $A \to [A/G]$ and $D \to [D/G]$ are regular coverings, so $\pi_i([A/G]) \to \pi_i([D/G])$ is an isomorphism for $i<k-1$ and surjective for $i=k-1$; hence $([D/G],[A/G])$ is $k-1$-connected.  The interiors of $\X \cup_{[S/G]} [A / G]$ and $[D/G]$ cover $\X^\prime$, so excision for homotopy groups implies that $(\X^\prime, \X \cup_{[S/G]} [A/G])$ is $k-1$-connected as well.

Finally, the equivariant map $D \to D$ compressing $A$ onto the boundary sphere $S$ is equivariantly homotopic to the identity through a homotopy fixing $S$.  Thus using the universal property of $\X \cup_{[S/G]} [D/G]$ we have a homotopy equivalence of pairs $(\X^\prime, \X \cup_{[S/G]} [A/G]) \simeq (\X^\prime, \X)$, so we deduce $\pi_i(\X^\prime, \X) = 0$ for $i<k$.
\end{proof}

Recall that $H_\ast(\real{\X},\real{\Z}) = H^\ast(\X,\Z)$ and likewise for $H^\ast$, so

\begin{cor} \label{cor_homological_DM_LHT}
For the relative homology and cohomology, with integer coefficients, we have $H^i(\X,\Z) = H_i(\X,\Z) = 0$ for $i<k$.  As a consequence $H^i(\X) \to H^i(\Z)$ is an isomorphism for $i<k-1$ and injective for $i=k-1$.
\end{cor}
\vskip 10 pt

\subsection{Topology of the underlying space}

For completeness we describe the consequences for the topology of the underlying spaces $\bar{\Z} \subset \bar{\X}$.  In the introduction we mention that the homological \LHT\ holds with rational coefficients for $\bar{\Z} \subset \bar{\X}$, but in fact we only need to invert the orders of the automorphism groups of points of $\X$.

\begin{prop} \label{thm_main_underlying_space}
The underlying space $\bar{\X}$ can be obtained from $\bar{\Z}$, up to homotopy equivalence, by attaching finitely many disc quotients $D^r/G$ along the boundary $S^{r-1}/G$.  Here, as before, $r \geq k$ and $G = \Aut_c$ acting linearly on the unit disc $D^r \subset T_c \X_- \subset T_c \X$.
\end{prop}

Using the Morse function $f$ constructed above, we know that $\bar{\X}^{\epsilon]}$ deformation retracts onto $\bar{\Z}$ for sufficiently small $\epsilon$.  From this point the proposition is a direct application of the Morse lemmas 7.5-7.7 in \cite{Hepworth09}, using the estimate on the index of the critical points of $f$.

\begin{cor}
$H_i(\bar{\X},\bar{\Z}; \ZZ[1/N]) = H^i (\bar{\X},\bar{\Z}; \ZZ[1/N]) = 0$ for $i<k$ as long as $N$ is a common multiple of the orders of the isotropy groups of $\X$ (of which there are finitely many because $\bar{\X}$ is compact).
\end{cor}

\begin{proof}
This follows from the fact that $H_i(D/G,S/G ; \ZZ[1/N]) = 0$ for $i<k$, which is really fact about group homology.  One could also deduce the claim directly from Corollary \ref{cor_homological_DM_LHT} and the isomorphism $H_\ast (\X ; \ZZ[1/N]) \to H_\ast (\bar{\X} ; \ZZ[1/N])$.
\end{proof}

\section{Morse theory of Deligne-Mumford stacks} \label{morse_theory_DM_stacks}

Now we will prove the main theorem of the Morse theory for \DM\ stacks used in the proof of Theorem \ref{thm_main_stacks}.  For this section we will take as input a function $f : \X \to \RR$ with the following properties:
\begin{itemize}
\item $f$ is nonnegative with a global minimum $\Z = f^{-1} \{0\} \subset \X$ along which $f$ satisfies a {\L}ojasiewicz inequality \eqref{eqn_lojasiewicz}.
\item $\bar{f} : \bar{\X} \to \RR$ is proper
\item $f$ is Morse on the complement $\X - \Z$, with critical points $c_1,c_2,\ldots,c_r$.
\end{itemize}
As with manifolds, the critical point data of $f$ determines the ``strong homotopy type'' of $\X$.

We fix a Riemannian metric $g$ on $\X$, and we will study the flow along the vector field $-\nabla f$.\footnote{On a \DM\ stack a vector field corresponds to a consistent choice of vector field $\xi_U$ on $U$ for all \'{e}tale $U \to \X$.  The tangent bundle of an Artin stack is more complicated \cite{Hepworth09b}}
\begin{fact}
Given a vector field $\xi$ on $\X$ with compact support, there is a representable morphism $\Phi : \X \times \RR \to \X$ such that for any \'{e}tale map $M \to \X$, the base change
$$\xymatrix{M^\prime \ar[d]_{\Phi^\prime} \ar[r] & \X \times \RR \ar[d]_{\Phi} \\ M \ar[r] & \X }$$
satisfies $D \Phi^\prime (\frac{\partial}{\partial t})_{M^\prime} = \xi _M$.  Furthermore there is an isomorphism $e_{\Phi} : \Phi|_{\X \times \{0\}} \Rightarrow \id_{\X}$, and the flow morphism $\Phi$ is uniquely determined by $e_{\Phi}$.
\end{fact}
The map $\Phi$ is group action of $\RR$ up to 2-isomorphism, and so  it induces an $\RR$ action on the underlying space $\bar{\X}$ and an $\RR$ action up to homotopy on the classifying space $\real{\X}$.

We will use existence of flows to deconstruct the topology of $\X$ in three steps.  First we show that $\X^{\epsilon]}$ deformation retracts onto $\Z$ for sufficiently small $\epsilon$.  Then we show that if the interval $[a,b]$ contains no critical values of $f$, then $\X^{b]}$ deformation retracts onto $\X^{a]}$, and finally we show how the topology changes as $a$ crosses a critical value.

First of all what does a deformation retract mean for stacks?  Say $j : \Z \to \X$ is a closed immersion, then a \textit{deformation retract} consists of a homotopy $h : \X \times I \to \X$ and a projection $\pi : \X \to \Z$, along with 2-morphisms
\[\xymatrix{\Z \times I \ar[r]^{j \times \id} \ar[d]_{\pr_1} & \X \times I \ar[d]^h & & \X \ar@/_/[r]_{\id} \ar@/^/[r]^{h_0} \ar@{}[r] |\Downarrow & \X \\ \Z \ar[r]_j \ar@{=>}[ur] & \X & &  \X \ar@/^/[r]^{h_1} \ar@/_/[r]_{j \circ \pi} \ar@{}[r] |\Downarrow & \X }\]
This definition is a categorification of the definition of deformation retract of spaces (or what's called a strong deformation retract by some authors), and implies that $j$ is a homotopy equivalence.  Both the geometric realization functor and the underlying space functor respect products, so if $\X$ deformation retract onto $\Z$ in the above sense, then $\bar{\X}$ deformation retracts onto $\bar{\Z}$ and $\real{\X}$ is homotopy equivalent to $\real{\Z}$.

\begin{rem} \label{rem_strict_retract}
A strict deformation retract of topological groupoids induces a deformation retract of stacks in the above sense.
\end{rem}

\begin{rem} \label{rem_easy_retract}
A morphism factoring through a substack does so uniquely up to unique isomorphism, so the morphism $\pi$ could be omitted from the definition of a deformation retract, instead just requiring that $h_1 : \X \to \X$ factor through $\Z \hookrightarrow \X$.

Thus if $h: \X \times I \to \X$ is a deformation retract of $\X$ onto $\Z$, and $\X^\prime \subset \X$ is a substack (not necessarily closed) containing $\Z$ such that $h|_{\X^\prime \times I} : \X^\prime \times I \to \X$ factors through $\X^\prime$, then $\X^\prime$ deformation retracts onto $\Z$ as well via this restricted $h$.
\end{rem}

\begin{rem}
For an embedding of topological stacks $j : \Z \to \X$, we choose a presentation for $\X$ and form the mapping cylinder $\stack{M}(j) = [M(j_1) \rightrightarrows M(j_0)] \subset \X \times [0,1]$.  Morita equivalent presentations give Morita equivalent mapping cylinders because the mapping cylinder construction for spaces respects fiber products and epimorphisms.  Furthermore, $\stack{M}(j) \cong \X \times \{0\} \cup_{\Z \times \{0\}} \Z \times [0,1]$ is a pushout in the 2-category of topological stacks. Thus we can simplify our definition of a deformation retract further as \emph{a map $h:\X \times I \to \X$ such that $h_1$ factors through $\Z$, along with an isomorphism $h|_{\stack{M}(j)} \cong \pr_1 |_{\stack{M}(j)}$}.
\end{rem}

We are now ready to prove the

\begin{lem} \label{lem_neighborhood_retract}
For $\epsilon > 0$ sufficiently small, $\X^{\epsilon]}$ contains no critical points of $f$, and the gradient descent flow of $f$ provides a deformation retract of $\X^{\epsilon]}$ onto $\Z$.
\end{lem}

\begin{proof}
The key idea is to find, for $\epsilon$ small enough, an atlas for $\X^{\epsilon]}$ that is \emph{equivariant} for the $\RR$ action.  Start with an \'{e}tale atlas $a: X_0 \to \X$.  $f|_{X_0}$ still satisfies the inequality \eqref{eqn_lojasiewicz} at every point in $Z_0 = (f_{X_0})^{-1} \{0\}$.

By Corollary \ref{cor_stability_gradient_flow} there is an open subset $U_0 \subset X_0$ containing $Z_0$ such that $(-\nabla f)_{X_0}$ is integrable to a flow $\Phi_0 : U_0 \times [0,\infty) \to U_0$ which extends uniquely to a deformation retract $h_0 : U_0 \times [0, \infty] \to U_0$ onto $Z_0$.

The \'{e}tale map $U_0 \to \X$ is equivariant with respect to the action of the semigroup $[0,\infty)$, so $[0,\infty)$ acts on $U_1 := U_0 \times_{\X} U_0$ as well, and we would like to lift the action to $\infty$, i.e. find a lift
$$\xymatrix{ U_1 \times [0,\infty) \ar[rr]^{\Phi_1} \ar@{^{(}->}[d] & & U_1 \ar[d]^{(s,t)} \\ U_1 \times [0,\infty] \ar[r]^-{(s,t)} \ar@{-->}@/^/[urr]^{\exists ! h_1} & U_0 \times U_0 \times [0,\infty] \ar[r]^-{h_0 \times h_0} & U_0 \times U_0}$$
One shows the existence and uniqueness of $h_1$ using the fact that, because $\X$ is \DM, $U_1 \to U_0 \times U_0$ is a proper immersion.

It suffices to show existence and uniqueness of $h_1$ in a neighborhood of each point $(\infty,p) \in [0,\infty] \times U_1$.  If $(x,y) \in U_0 \times U_0$ is the image of $(\infty,p)$, then there is a compact neighborhood $N$ of $(x,y)$ such that $(s,t)^{-1} (N)$ is a disjoint union of connected compact subsets, each mapped homeomorphically onto its image by $(s,t)$.  Finally for a neighborhood of $(\infty, p)$ whose image in $U_0 \times U_0$ is contained in $N$ there exists a unique lift.

Thus we get a map of groupoids $h : [0,\infty] \times (U_1 \rightrightarrows U_0) \to (U_1 \rightrightarrows U_0)$ providing a deformation retract of $U_1 \rightrightarrows U_0$ onto the subgroupoid $Z_1 \rightrightarrows Z_0$.  By Remark \ref{rem_strict_retract} $h$, gives a deformation retract of the open substack $\mathfrak{U} = [U_1\rightrightarrows U_0]$ onto $\Z$.  Because $\bar{\Z}$ is compact, $\X^{\epsilon]}$ is contained in $\mathfrak{U}$ for sufficiently small $\epsilon$, so by Remark \ref{rem_easy_retract} the result follows.
\end{proof}

Next we study what happens when there are no critical points between $a$ and $b$.
\begin{prop} \label{morse_no_crit}
If $f$ has no critical points in $\X^{[a,b]}$, then there is a deformation retract of $\X^{b]}$ onto $\X^{a]}$.
\end{prop}

\begin{proof}
Let $\xi$ be a vector field with compact support such that $\xi \cdot f \leq 0$ and $\xi \cdot f = -1$ on $\X^{[a,b]}$.  Let $\Phi : \X \times \RR \to \X$ be the flow of $\xi$ and form the composition
$$h : \X^{b]} \times [0,1] \xrightarrow{\left( \id_{\X^{b]}}, \max(0, t (f-a)) \right) } \X^{b]} \times [0,\infty) \xrightarrow{\Phi} \X$$
$\bar{f}$ decreases along the flow lines of $\bar{\Phi}$ and decreases with constant rate $1$ in $\bar{\X}^{[a,b]}$, so it follows that $h : \X^{b]} \times [0,1] \to \X$ factors through $\X^{b]}$ and that $h_1$ factors through $\X^{a]}$.\footnote{A map $\psi : \stack{Y} \to \X$ factors through $\X^{b]}$ if and only if $f \circ \psi : \stack{Y} \to \RR$ factors through $(-\infty,b]$, so it suffices to check on the level of underlying spaces.} The isomorphism $e_\Phi : \Phi|_{\X \times \{0\}} \Rightarrow \id_\X$ induces an isomorphism $h|_{\stack{M}(j)} \simeq \pr_1|_{\stack{M}(j)}$ because $\max (0, t (f-a)) = 0$ on the substack $\stack{M}(j) \subset \X^{b]} \times I$.
\end{proof}

\begin{rem}
An argument identical to that in theorem 3.1 of \cite{Milnor63} or theorem 7.5 of \cite{Hepworth09} can also be used to show that $\X^{a]} \simeq \X^{b]}$ as stacks.
\end{rem}

\begin{prop}
Let $c \in \RR$ be a critical value of $f$, and assume that $f$ is Morse near $f^{-1}\{c\}$.  Denote the critical points in $f^{-1} \{c\}$ by $p_1,\ldots,p_r$.  Then for $\epsilon$ sufficiently small there is a closed topological substack
$$\X^{c-\epsilon]} \subset \X^\prime \subset \X^{c+\epsilon]}$$
such that $\X^{c+\epsilon]}$ deformation retracts onto $\X^\prime$, and
$$\X^\prime \simeq \X^{c-\epsilon]} \cup_{\coprod_i [S_i / \Aut_{p_i} ]} \coprod_i [ D_i / \Aut_{p_i} ],$$
where $D_i$ and $S_i$ are the disc and sphere of radius $\epsilon$ in the index representation of $\Aut_{p_i}$.
\end{prop}

\begin{proof}
Again, the idea of the proof is identical to the discussion in \cite{Milnor63} or \cite{Hepworth09}, one just has to check that the argument gives a cell attachment of \textit{stacks} and not just underlying spaces.

For each $p_i$, identify an open substack containing $p_i$ with $[U_{p_i} / \Aut p_i]$, where $U_{p_i}$ is a ball around the origin in $T_{p_i} \X$, and coordinates have been chosen in $T_{p_i} \X = T_{p_i} \X_+ \oplus T_{p_i} \X_-$ such that the function $f|_{U_i} = c + |u_+|^2 - |u_-|^2$.  Choose an $\epsilon$ smaller than the radii of all of these Morse coordinate patches, and define the closed substack of $\X^{c+\epsilon]}$
$$\X^\prime = \{f \leq c-\epsilon\} \cup \bigcup_{i} \{u_+ =0 \text{ and } |u_-|^2 \leq \epsilon\}$$
Where we have used the slightly informal notation $\{u_+ =0 \text{ and } |u_-|^2 \leq \epsilon\}$ for the closed substack $[D_i / \Aut_{p_i}]$.

Theorem 16.9 of \cite{Noohi05} gives a criterion for checking that $\X^\prime$ is in fact a 2-categorical union.  First of all, $\X^\prime$ is manifestly a gluing in the language of \ibid, in that $\X^\prime \setminus \coprod_i [S_i/\Aut_{p_i}] \simeq  \X^{c-\epsilon]} \sqcup \coprod_{i} [(D_i\setminus S_i) / \Aut_{p_i}]$.  The second criterion -- the existence of an atlas $a : X^\prime_0 \to \X^\prime$ such that the pair of invariant subspaces $a^{-1} ([S_i/\Aut_{p_i}]) \subset a^{-1}([D_i/ \Aut_{p_i}])$ satisfies the local left lifting property (LLLP) with respect to the class of covering maps -- is tautologically satisfied because every pair satisfies the LLLP with respect to local homeomorphisms.\footnote{Although it is not necessary here, one could even find an atlas $a:X_0^\prime \to \X^\prime$ for which $a^{-1}([D_i/ \Aut_{p_i}]) = D_i$ and likewise for $S_i$}  So we have really verified that a union of two substacks of a \DM\ stack is always a 2-categorical union.

The rest of the argument for proceeds exactly as in \cite{Milnor63} and \cite{Hepworth09}.  One introduces an auxiliary function $F$ such that: $F$ agrees with $f$ outside of the Morse coordinate patches (on which $F\leq f$), $\{F\leq c+\epsilon \} = \X^{c+\epsilon]}$, and $F(p_i) < c-\epsilon$.  Then one uses Proposition \ref{morse_no_crit} to retract $\X^{c+\epsilon]}$ onto $\{F \leq c-\epsilon\}$, a substack which differs from $\X^{c-\epsilon}$ only within the Morse coordinate patches.  Finally one constructs a manifestly $\Aut_{p_i}$ equivariant deformation retract of $\{ F \leq c-\epsilon \}$ onto $\X^\prime$ in each $[U_i / \Aut_{p_i}]$ which restricts to the identity on $\X^{c-\epsilon]} \cap [U_i / \Aut_{p_i}]$.  These can then be glued to give the global deformation retract of $\{F \leq c-\epsilon\}$ onto $\X^\prime$.
\end{proof}

For reference we now summarize the main theorem of Morse theory for compact Deligne-Mumford stacks which we have just proven
\begin{thm}
Let $f : \X \to \RR$ be a smooth nonnegative function on a compact Deligne-Mumford stack $\X$ achieving its global minimum on the closed topological substack $f^{-1} \{0\}$.  Assume that
\begin{itemize}
\item $f$ satisfies a {\L}ojasiewicz inequality \eqref{eqn_lojasiewicz} at every $p \in \bar{\Z}$, and
\item $f$ is Morse on the open substack $\X - \Z$.
\end{itemize}
Then $\X$ can be obtained from $\Z$ by a finite sequence of deformation retracts and 2-categorical attachments of $[D_i / \Aut_{p_i} ]$ along $[S_i / \Aut_{p_i}]$ where $p_i$ ranges over the critical points of $f$ in $\bar{\X} - \bar{\Z}$ and $D_i \subset T_{p_i} \X_-$ is a small disc with boundary $S_i$.
\end{thm}

\section{Homological Lefschetz theorem for Artin stacks} \label{homological_artin_LHT}
The \LHT\ for Artin\footnote{An Artin stack is just a differentiable stack admitting a surjective submersion from a manifold} stacks is beyond the current reach of Morse theoretic techniques.  The theory of integration of vector fields on Artin stacks has been worked out in \cite{Hepworth09b}, but Morse theory has not been developed yet.

In this section we describe a local-to-global proof of the homological \LHT\ that applies to differentiable Artin stacks admitting a presentation by compact manifolds.

In the complex category, proper Artin stacks are close to being gerbes, so in order to have a version of the \LHT\ with broader applications, we will leave the category of complex stacks.  For instance, the following theorem applies to a global quotient of a compact complex manifold by a compact Lie group.

\begin{thm} \label{thm_artin_LHT}
Let $\LL$ be a hermitian line bundle on a differentiable Artin stack $\X$, and let $s$ be a section.  Assume that there exists a proper submersion $X_0 \to \X$ from a compact complex manifold $X_0$ such that the pullback $L_0 = \LL|_{X_0}$ is holomorphic of type $k$ and the section $s|_{X_0}$ is holomorphic.

As before let $\Z \subset \X$ be the topological substack on which $s$ vanishes. Then $H^i(\X,\Z) = H_i(\X,\Z) = 0$ for $i<k$.
\end{thm}

\begin{proof}
From the proper submersion $X_0 \to \X$ we get a closed, full, saturated sub-groupoid $Z_\bullet \subset X_\bullet$ presenting $\Z \subset \X$.  Let $\{X_p\}$ and $\{Z_p\}$ denote the respective simplicial nerves.

We can form the smooth function $f = |s|^2$ on $\X$, and as before $\Z$ is precisely the zero locus of $f$.  In particular $Z_p$ is the zero locus of $f|_{X_p} = (f|_{X_0})|_{X_p}$.  Additionally $Z_p$ is the preimage of $Z_0$ under any of the $p+1$ simplicial maps $X_p \to X_0$.

Now the function $f|_{X_0}$ is the norm squared of a section of a hermitian line bundle of type $k$.  Thus as in Section \ref{LHT_DM_stacks} above we can perturb $f|_{X_0}$ to a function which we call $\phi_0$ which
\begin{itemize}
\item satisfies a {\L}ojasiewicz inequality near its global minimum $Z_0$,
\item is Morse(-Bott) away from $Z_0$ with critical points of index $\geq k$.
\end{itemize}
Define the functions $\phi_p = \phi_0 |_{X_p}$ using the $0^{th}$ simplicial face map $X_p \to X_0$ to restrict to $X_p$.  Note that $\phi_0$ is no longer invariant (does not necessarily descend to a smooth function on $\X$), so we must be explicit about which simplicial maps we use.  Both bulleted properties are preserved when one restricts along a surjective submersion, so the $\phi_p$ are functions with {\L}ojasiewicz global minima $Z_p = \phi_p^{-1} \{0\}$ that are Morse-Bott away from $Z_p$ with critical manifolds of index $\geq k$.  It follows from non-degenerate Morse theory that $H_i(X_p,Z_p) = H^i(X_p,Z_p) = 0$ for all $i<k$.

From the bicomplex computing $H_\ast (\X,\Z)$ we have a homological spectral sequence
\begin{equation}\label{eqn_loc_to_glob}
E^1_{p,q} = H_q (X_p, Z_p) \Rightarrow H_{p+q}(\X, \Z)
\end{equation}
so the vanishing of $H_i(X_p,Z_p)$ implies $H_i(\X,\Z) = 0$ for $i<k$.  A similar spectral sequence implies the result for cohomology.
\end{proof}

\begin{ex}
This theorem applies to global quotients even when the group does not act holomorphically.  For instance let $G = \ZZ / 2 \ZZ$ acting on $\CP^l$ by complex conjugation $[z_0 : \cdots : z_l] \mapsto [\bar{z}_0 : \cdots : \bar{z}_l]$.  By identifying $\CC^{l+1} = \CC \otimes \RR^{l+1}$, real subspaces of $\RR^{l+1}$ induce complex subspaces of $\CC^{l+1}$, and thus we have an embedding $\RP^l \subset \CP^l$ as the fixed locus of $G$.

Complex conjugation acts equivariantly on $\OO_{\CP^l}(1)$, and an invariant holomorphic section corresponds to a form with real coefficients $s = r_0 z_0 + \cdots r_l z_l$.  The zero locus of $s$ is a hypersurface $H \subset \CP^l$ induced by a real hypersurface in $\RR^{l+1}$, and by a real coordinate change we may assume that $H = \{ [0 : \ast : \cdots \ast ] \}$.

Let $p = [1:0:\cdots:0] \in \CP^l$.  Then $\CP^l - \{p\}$ deformation retracts equivariantly onto $H$, so we have
\begin{align*}
H^i_G (\CP^l, H) &= H^i_G (\CP^l, \CP^l - \{p\}) = H^i_G (\CC^l, \CC^l - \{0\}) \\
&\simeq H^{i-l} ( \real{[\ast / \ZZ/2]} ; \Omega )
\end{align*}
In this last expression, we have used the Thom isomorphism for the bundle $[\CC^l / G] \to [\ast/G]$, and $\Omega$ is a local system on $\real{[\ast/G]}$ which depends on the orientability of this bundle.  Thus we have $H^i_{G}(\CP^l, H) = 0$ for $i<l$.
\end{ex}

\begin{ex}
Non-holomorphic group actions also arise naturally as subgroups of the group of unit quaternions acting by left multiplication on $\mathbb{P}(\mathbb{H}^{l})$, where choose either $i,j,$ or $k$ as the complex structure on $\mathbb{H}^l$.
\end{ex}

\begin{ex}
Let $X$ be a projective variety of dimension $n$ and $\X \to X$ a gerbe for a compact group $G$, and let $Z \subset X$ be a hyperplane section.  The theorem implies that the cohomology of $\X$ agrees with the cohomology of the restricted gerbe $\Z \to Z$ in degree $<n-1$.

In this case we can also see this by comparing the Leray-Serre spectral sequence $E_2^{p,q} = H^p (X; H^q(BG)) \Rightarrow H^{p+q}(\X)$ and the corresponding sequence for $\Z \to Z$.  The result follows because, by the classical \LHT, the restriction map on the $E^2$ page $H^p(X;H^q(BG)) \to H^p(Z;H^q(BG))$ is an isomorphism for $p<n-1$.
\end{ex}

Unlike Theorem \ref{thm_artin_LHT}, the Leray-Serre argument in the example above does not require $G$ to be compact, which suggests that Theorem \ref{thm_artin_LHT} might hold for a much larger class of Artin stacks.  We also observe that the homological \LHT\ also applies to global quotients by reductive groups, giving further evidence for a more general statement than Theorem \ref{thm_artin_LHT}.
\begin{cor}
Let $G$ be a compact Lie group and let the complexification $G_{\CC}$ act on a compact complex manifold $X$, and let $(\LL,\h)$ be a $G_\CC$-equivariant hermitian line bundle of type $k$ with invariant global section $s$ vanishing on $Z \subset X$, then $H_i([X/G_{\CC}] , [Z/G_{\CC}]) = H^i([X/G_{\CC}],[Z/G_{\CC}]) = 0$ for $i<k$.
\end{cor}

\begin{proof}
$G_{\CC}$ deformation retracts onto $G$, and so the pair $(X \times (G_{\CC})^{\times p}, Z \times (G_{\CC})^{\times p})$ deformation retracts onto $(X \times (G)^{\times p},Z \times (G)^{\times p})$ for any number of factors $p$.  Thus by the spectral sequence \eqref{eqn_loc_to_glob}, $H_i( [X/G] , [Z/G] ) \to H_i( [X/G_{\CC}], [Z/G_{\CC}])$ is an isomorphism, and $H_i([X/G] , [Z/G] ) = 0$ for $i<k$ by Theorem \ref{thm_artin_LHT}.
\end{proof}

\appendix
\section{Stability of gradient descent flow} \label{app_lojasiewicz}
We have used Morse-theoretic techniques to prove the \LHT\ without requiring the section of the line bundle to be non-degenerate as in \cite{Bott59}.  The global statements in the text above rely on the local results summarized in this appendix.  Here we prove the key fact: that the gradient descent flow of a function near its global minimum is \emph{stable} as long as it satisfies the following property

\begin{defn}
Let $X$ be a Riemannian manifold and $f$ a smooth function on $X$.  We sat that $f$ satisfies a \textit{{\L}ojasiewicz inequality} at a point $p$ if there are constants $\rho \in (0,1)$ and $C>0$ such that
\begin{equation*} \tag{L} \label{eqn_lojasiewicz}
|f - f(p)|^\rho \leq C | \nabla f|
\end{equation*}
in some neighborhood of $p$
\end{defn}

As a consequence of \eqref{eqn_lojasiewicz}, $f$ takes the value $f(p)$ at any critical point near $p$.  In particular if $Z$ is the global minimal set of $f$ and \eqref{eqn_lojasiewicz} holds at every point of $Z$, then $Z$ is isolated from the rest of the critical locus of $f$.
\begin{prop} \label{prop_stability_gradient_flow}
If $f$ achieves its global minimum along $Z \subset X$ and \eqref{eqn_lojasiewicz} holds at $p \in Z$, then there is an open set $U \subset X$ containing $p$ such that
\begin{enumerate}
\item the flow $\phi_t(x)$ of $-\nabla f$ is defined for all $x \in U$ and all $t\geq 0$, and $\phi_t(x) \in U$.
\item The map $U \times [0,\infty) \to U$ given by $(x, t) \mapsto \phi_t(x)$ extends uniquely to a continuous map $U \times [0,\infty] \to U$.  In particular the flow of $-\nabla f$ deformation retracts $U$ onto $U \cap Z$.
\end{enumerate}
\end{prop}

\begin{rem}
The proposition is equally valid for a local minimum.  Also, the argument below shows that any gradient like vector field for $f$ will have the same stability property.
\end{rem}

\begin{proof}

For simplicity we will assume that $f\geq 0$ and $f(p) = 0$.  Let $\phi_t(x_0)$ be an integral curve starting at $x_0 \in X$.  As long as $\phi_t(x_0)$ stays in a region in which \eqref{eqn_lojasiewicz} holds, an arc length integral gives a uniform bound
\begin{equation} \label{eqn_dist_estimate}
\operatorname{dist} (x_0, \phi_t(x_0)) \leq \frac{C}{1-\rho}f(x_0)^{1-\rho} \qquad \text{for }t\geq0
\end{equation}

Start with a relatively compact open ball $B_{r^\prime}$ of radius $r^\prime$ around $p$ on which the inequality \eqref{eqn_lojasiewicz} holds.  Now we can find a smaller ball $B_r \subset B_{r^\prime}$ such that the difference in radii $r^\prime - r$ is larger than $C f(x_0)^{1-\rho}/(1-\rho)$ for any point $x_0 \in B_r$.  It follows from \eqref{eqn_dist_estimate} that any flow line starting in $B_{r}$ stays within $B_{r^\prime}$ under the (positive time) flow of $-\nabla f$, and thus by the escape lemma \cite{Lee03} the flow is defined on $B_r$ for all $t\geq 0$.  We let $U = \bigcup_{t\geq 0} \phi_t (B_r)$.

Inequality \eqref{eqn_lojasiewicz} implies that $f(\phi_t(x_0)) \to 0$ as $t \to \infty$.  Combining this with \eqref{eqn_dist_estimate} shows that $\phi_t (x_0)$ is Cauchy as $t\to \infty$ and remains in a compact region, so $\phi_\infty (x) = \lim_{t\to \infty}\phi_t (x)$ is a well defined function on $U$.  In fact \eqref{eqn_dist_estimate} shows that $(x,t) \mapsto \phi_t(x)$ is uniformly continuous as $t \to \infty$, and so the gradient descent flow extends uniquely to a continuous map $U \times [0,\infty] \to U$. Finally, $\phi_s( \phi_\infty (x)) = \lim_{t\to \infty} \phi_{s+t}(x) = \phi_\infty (x)$, so $\phi_\infty(U) \subset Z$ and the map constructed above is a deformation retract of $U$ onto $U\cap Z$.

\end{proof}

An immediate corollary of Proposition \ref{prop_stability_gradient_flow} is the global statement that
\begin{cor} \label{cor_stability_gradient_flow}
Let $f$ be a smooth nonnegative function on $X$ with global minimum $Z = f^{-1}\{0\} \subset X$.  If the inequality \eqref{eqn_lojasiewicz} holds at each $p\in Z$, then there is an open neighborhood $U$ of $Z$ on which the negative gradient flow extends uniquely to a map $U \times [0,\infty] \to U$ which is a deformation retract of $U$ onto $Z$.
\end{cor}

In addition if $Z$ is compact then $X^{\epsilon]}$ deformation retracts onto $Z$ for sufficiently small $\epsilon$.

We have shown above that the {\L}ojasiewicz inequality implies a stable gradient flow, but we are left with the question of which functions satisfy \eqref{eqn_lojasiewicz}.  First of all, functions satisfying \eqref{eqn_lojasiewicz} have the following properties.

\begin{enumerate}
\item if $h$ is a smooth function with $h(p) \neq 0$ and $f$ satisfies \eqref{eqn_lojasiewicz} with $f(p) = 0$, then so does $h f$\\
\item Let $f$ and $g$ satisfy \eqref{eqn_lojasiewicz} with $f(p)=g(p)=0$.  If the angle between $\nabla f$ and $\nabla g$ is bounded away from $\pi$ near $p$ then $f + g$ satisfies \eqref{eqn_lojasiewicz},\footnote{More precisely the condition is $(\nabla f, \nabla g) \geq (-1+\epsilon) |\nabla f| |\nabla g|$ in a neighborhood of $p$ for some small $\epsilon$} and if the angle between $\nabla f^2$ and $\nabla g^2$ is bounded away from $\pi$ then $f g$ satisfies \eqref{eqn_lojasiewicz}.\\
\item if $\pi : X \to Y$ is an open mapping near $p\in X$ and $f \circ \pi$ satisfies \eqref{eqn_lojasiewicz} at $p$ then $f$ satisfies \eqref{eqn_lojasiewicz} at $\pi(p) \in Y$.  If $\pi$ is submersive at $p$, then the converse is true as well.  Letting $\pi$ be the identity map shows that \eqref{eqn_lojasiewicz} is independent of the metric. \\
\item if $\pi : X \to Y$ is proper and surjective onto a neighborhood of $q \in Y$ and $f \circ \pi$ satisfies \eqref{eqn_lojasiewicz} at every point in the fiber $\pi^{-1} \{q\}$, then $f$ satisfies \eqref{eqn_lojasiewicz} at $q$.
\end{enumerate}

One might wonder in light of property (2) above if arbitrary algebraic combinations of functions will still satisfy \eqref{eqn_lojasiewicz}, but this turns out to be false.  For instance if $\rho : \RR \to \RR$ vanishes to all orders at $0$, then $f(x) = (1+\rho) x - x = \rho x$ does not satisfy \eqref{eqn_lojasiewicz} at $0$ even though $(1+\rho)x$ and $x$ do.  Thus while inequality \eqref{eqn_lojasiewicz} has convenient dynamical properties, it lacks algebraic properties.  It is thus surprising that many functions do satisfy the {\L}ojasiewicz inequality:

\begin{thm} \label{app_stab_class}
Let $f$ be smooth function on $X$ which in some choice of local coordinates is a non-vanishing smooth function times a real analytic function which vanishes at $p$, then $f$ satisfies property \eqref{eqn_lojasiewicz} at $p$.
\end{thm}

\begin{proof}
Using real analytic resolution of singularities there is a proper surjective analytic map $\pi : Y \to X$ such that $f \circ \pi = h(y) y_1^{\alpha_1} \cdots y_n^{\alpha_n}$ where $y_1,\ldots,y_n$ are local analytic coordinates on $Y$ and $h$ is a smooth function which does not vanish near $\pi^{-1} \{p\}$.  Applying property (2) above, or by a direct calculation, one shows that the monomial $y_1^{\alpha_1} \cdots y_n^{\alpha_n}$ satisfies \eqref{eqn_lojasiewicz} along its zero locus, and thus by property (1) so does $f \circ \pi$.  Finally by property (4) $f$ satisfies \eqref{eqn_lojasiewicz} at $p$.
\end{proof}

{\L}ojasiewicz was the first to prove that real analytic functions satisfy the inequality \eqref{eqn_lojasiewicz} and apply it to the topology of real-analytic varieties.  Many authors have extended his work, for instance in \cite{KMP00}.

\bibliographystyle{amsalpha}
\bibliography{LHT_bibliography}

\end{document}